\documentclass{amsart}
\usepackage{geometry} 
\usepackage{amssymb}
\usepackage{faktor}
\usepackage[all,cmtip]{xy}
\usepackage{amsmath,amscd}
\usepackage{graphicx}
\usepackage{mathrsfs}
\usepackage{color}
\usepackage{bbm}
\usepackage[mathcal]{euscript}
\usepackage{hyperref}
\hypersetup{colorlinks=true,linkcolor=blue}

\newcommand{\End}{\mathcal{E}\mathrm{nd}}
\newcommand{\HHom}{\mathcal{H}\mathrm{om}}
\newcommand{\Ho}{\mathrm{Ho}}

\newcommand{\Hom}{\mathrm{Hom}}

\newcommand{\Emb}{\mathrm{Emb}}
\newcommand{\Conf}{\mathrm{Conf}}

\newcommand{\pt}{\mathrm{pt}}

\newcommand{\Fr}{\mathrm{Fr}}

\DeclareMathOperator*{\hocolim}{\mathrm{hocolim}}
\DeclareMathOperator*{\colim}{\mathrm{colim}}
\newcommand{\bvt}{\star}

\newcommand{\mat}{\star}

\newcommand \vp{\varphi}
\newcommand \ve{\varepsilon}

\newcommand\ob{\operatorname{Ob}}

\newcommand\op{\mathcal}
\newcommand\cat{\mathsf}

\newcommand\adjunct[4]{\xymatrix{#1\ar @<1.25ex>[rr]^-{#3}&\perp&#2\ar @<1.25ex>[ll]^-{#4}}}

\newcommand\map{\operatorname{Hom}}

\newcommand\Mod[1]{\cat{Mod}_{#1}}
\newcommand\grMod[1]{\cat{grMod}_{#1}}
\newcommand\dgMod[1]{\cat{dgMod}_{#1}}

\newcommand\Ch[1]{\cat{Ch}_{#1}}

\renewcommand\d{\operatorname{diag}}

\renewcommand\S{\cat {sSet}}

\newcommand\F{\cat F}

\newcommand\V{\cat V}

\newcommand\Mfld{\cat{Mfld}}
\newcommand\Disk{\cat{Disk}}

\def\treeof(#1;#2){[#1;#2]}

\newcommand{\comp}{\relax}
\def\comp(#1;#2){#1\circ(#2)}

\theoremstyle{remark}
\theoremstyle{definition}\newtheorem{definition}{Definition}[section]
\theoremstyle{theorem}\newtheorem{lemma}[definition]{Lemma}
\theoremstyle{theorem}\newtheorem{hypothesis}[definition]{Hypothesis}

\theoremstyle{remark}\newtheorem{remark}[definition]{Remark}
\theoremstyle{definition}\newtheorem{notation}[definition]{Notation}
\theoremstyle{definition}\newtheorem{warning}[definition]{Warning}
\theoremstyle{definition}
\theoremstyle{definition}
\theoremstyle{definition}\newtheorem{example}[definition]{Example}
\theoremstyle{theorem}\newtheorem{proposition}[definition]{Proposition}
\theoremstyle{theorem}\newtheorem{corollary}[definition]{Corollary}
\theoremstyle{theorem}\newtheorem{theorem}[definition]{Theorem}
\theoremstyle{definition}
\theoremstyle{theorem}\newtheorem{conjecture}[definition]{Conjecture}

\title{A K\"{u}nneth theorem for configuration spaces}
\author{Kathryn Hess}
\address{EPFL SV UPHESS, Station 8, CH-1015 Lausanne, Switzerland}
\email{kathryn.hess@epfl.ch}
\author{Ben Knudsen}
\address{Department of Mathematics, Northeastern University, Boston 02115, USA}
\email{knudsen@math.harvard.edu}
\date{} 

\begin{document}

\maketitle

\begin{abstract}
We construct a spectral sequence converging to the homology of the ordered configuration spaces of a product of parallelizable manifolds. To identify the second page of this spectral sequence, we introduce a version of the Boardman--Vogt tensor product for linear operadic modules, a purely algebraic operation. Using the rational formality of the little cubes operads, we show that our spectral sequence collapses in characteristic zero.
\end{abstract}

\section{Introduction}

In this article, we study the singular homology of the \emph{ordered configuration space} of $k$ points in a manifold $X$, which is the space \[\Conf_k(X)=\left\{(x_1,\ldots, x_k)\in X^k: x_i\neq x_j\text{ if } i\neq j\right\}.\] Although these spaces have enjoyed a long history of study in algebraic topology \cite{fadell-neuwirth,may2,mcduff}, complete homology computations remain rare \cite{arnold,cohen-lada-may,FeichtnerZiegler:ICAOCSS}.

The current tool of choice for such computations is the Leray--Serre spectral sequence for the inclusion $\Conf_k(X)\subseteq X^k$, as studied in \cite{totaro}. The first nontrivial differential of this spectral sequence, although explicit, often presents a computation too forbidding to permit further progress, and the spectral sequence is known not to collapse in general \cite{FelixThomas:CSMP}.

Even the celebrated representation stability theorem of \cite{ChurchEllenbergFarb:FIMSRSG} does little to lighten this gloomy outlook; indeed, computing stable multiplicities of representations in $H_*(\Conf_k(X);\mathbb{Q})$ is a difficult open problem in almost all cases \cite[Prob. 3.5]{Farb:RS}. Even the multiplicity of the trivial representation, corresponding to the homology of unordered configuration spaces, was unknown for closed, orientable surfaces of positive genus until very recently---see \cite{Drummond-ColeKnudsen:BNCSS} for the general case and \cite{Maguire:CCCS,Schiessl:BNUCST} for two concurrent computations in the case of the torus.

The main contribution of this article is a new tool for attacking such computations, which is valid for manifolds of the form $X=M\times N$ with $M$ and $N$ parallelizable. Building on \cite{DwyerHessKnudsen:CSP}, we view the sequence $\Conf(M)=\big( \Conf_k(M)\big)_{k\geq 0}$ of configuration spaces of $M$ as a module over the little $m$-cubes operad $\op C_m$, and similarly for $N$ and $M\times N$. The main result is a kind of K\"{u}nneth decomposition  in terms of the \emph{linearized Boardman--Vogt tensor product} $\star$, introduced below in Section \ref{section:linearized tensor product}.

\begin{theorem}[``K\"{u}nneth'' spectral sequence]\label{thm:kuenneth ss}
Let $M$ and $N$ be parallelizable $m$- and $n$-manifolds, respectively, and $R$ a commutative ring such that $H_*(\Conf_k(M);R)$ and $H_*(\Conf_k(N);R)$ are $R$-projective for each $k\geq0$. There is a natural, convergent spectral sequence \[E_{p,q}^2\cong H_p\left(H_*\left(\Conf(M);R\right)\star^\mathbb{L} H_*\left(\Conf(N);R\right)\right)_q\implies H_{p+q}\left(\Conf(M\times N);R\right)\] of $R$-linear $\op C_{m+n}$-modules.
\end{theorem}

For many purposes, the utility of a spectral sequence is determined by one's knowledge---or, more typically, lack of knowledge---of its differentials. We resolve this difficulty over the rationals, which is already a case of great interest.

\begin{theorem}[Collapse]\label{thm:collapse}
If $R$ is a field of characteristic zero, then the spectral sequence of Theorem \ref{thm:kuenneth ss} degenerates at $E^2$.
\end{theorem}

The proof relies on Kontsevich's formality theorem, as recently extended by Fresse--Willwacher \cite{fresse-willwacher}, and applies to any field over which formality holds---at the time of writing, the formality question remains unanswered in odd characteristic. Both results have natural extensions to non-parallelizable manifolds, and we prove these more general statements assuming a certain additivity conjecture of \cite{DwyerHessKnudsen:CSP}---see Corollary \ref{cor:structured spectral sequence} and Theorem \ref{thm:structured collapse} below.

Theorems \ref{thm:kuenneth ss} and \ref{thm:collapse} reduce the computation of the rational homology of ordered configuration spaces of products---assuming knowledge of the factors---to a purely algebraic problem in the representation theory of certain combinatorial categories. For example, in the case $m=n=1$, which encompasses the motivating example of the torus, the relevant representation theory is that of Pirashvili--Richter's category of non-commutative finite sets \cite{pirashvili-richter}. We will return to these computations in the sequel to this paper.

\begin{remark}
With more care, one can construct a cohomology spectral sequence with $E_2$ describable in terms of a certain cotensor product of linear operadic comodules. We leave the details of this extension to future work or the enthusiastic reader.
\end{remark}

\begin{remark}
It is natural to imagine that the assumption of projectivity in Theorem \ref{thm:kuenneth ss} could be weakened to one of flatness by working with slightly different model structures. In practice, of course, the ring of interest is typically either a field or the integers, so the utility of such a weakening would likely be small.
\end{remark}

\subsection{Conventions} We adhere to the following throughout.
\begin{enumerate}
\item We work over a fixed commutative, unital ring $R$. 
\item We always consider the categories $\cat{sSet}$ of simplicial sets and $\Ch{R}$ of unbounded chain complexes of $R$-modules to be equipped with the standard Kan--Quillen and projective model structures, respectively.
\item Abusively, the phrase ``simplicial category" refers to a category enriched over $\cat {sSet}$.
\item We write $\F$ for the category of finite sets and functions between them, $\cat \Sigma\subseteq \F$ for the wide subcategory of bijections, and $\cat{N}\subseteq \F$ for the discrete subcategory consisting of the sets $\{1,\ldots, n\}$ for $n\geq0$.
\item Manifolds are implicitly smooth of finite type.
\item Topological groups are assumed to be locally compact and Hausdorff.
\item We write $\Sigma_I$ for the automorphisms of the finite set $I$ and set $\Sigma_n=\Sigma_{\{1,\dots, n\}}$.
\item Given a comonad $\mathbb K=(K,\Delta, \ve)$ on a category $\cat A$, we write $B_\bullet^{\mathbb K}: \cat A \to \cat A^{\cat \Delta ^{\mathrm{op}}}$ for the comonadic bar construction on $\mathbb{K}$. Explicitly, $B_n^{\mathbb K}(A)=K^{n+1}(A)$, and the faces and degeneracies are induced by the counit $\ve$ and comultiplication $\Delta$, respectively.
\item Variations on the Boardman--Vogt tensor product considered below will be denoted by the symbol $\star$ rather than $\otimes$. This notation is motivated by the desire to avoid confusion with other tensor products at play, but the reader is warned that it differs from that of most other references.
\end{enumerate}

\subsection{Acknowledgements} The authors thank the Isaac Newton Institute for Mathematical Sciences for support and hospitality during the programme Homotopy Harnessing Higher Structures. This work was supported by EPSRC grants EP/K032208/1 and EP/R014604/1 and NSF award 1606422.

\section{Operads and modules}

In this section, we recall the concepts from the theory of operads and operadic modules employed throughout the remainder of the article.

\subsection{Operads and modules}
We review here the fundamental definitions that we will use throughout the remainder of the paper. Throughout this section $(\cat \V, \otimes, \Hom_{\cat V}, 1_{\cat V})$ denotes a cocomplete closed symmetric monoidal category with an initial object, $\varnothing$. The two examples relevant for our purposes are the categories $\cat{sSet}$ of simplicial sets, equipped with the Cartesian monoidal structure, and $\cat{Ch}_R$ of unbounded chain complexes over a commutative ring $R$, equipped with the tensor product over $R$.

\begin{definition} A \emph{sequence} in $\cat V$ is a functor $\cat{N}\to \cat{V}$. A \emph{symmetric sequence} in $\cat V$ is a functor $\cat \Sigma^{op} \to \cat {V}$.
\end{definition}

A symmetric sequence is determined by the $\Sigma_n$-objects $\op X(n):=\op X\big(\{1,...,n\}\big)$ for $n\geq0$. The $\V$-category of symmetric sequences carries a (non-symmetric) monoidal structure called the \emph{composition product}, which is defined by the formula
\[(\op Y\circ\op X)(I)=\colim_{\cat \Sigma}\left(J\mapsto  \op Y(J)\otimes\bigoplus_{f\in \F(I,J)}\bigotimes_{j\in J}\op X(f^{-1}(j))\right),\]
where $\op X$ and $\op Y$ are symmetric sequences and $I$ is a finite set.   

\begin{definition}
An \emph{operad} in $\cat V$ is a monoid in $(\cat {\cat V}^{\cat \Sigma^{op}},\circ)$. A map of operads is a map of monoids.
\end{definition}

We write $\cat {Op}(\cat V)$ for the category of operads in $\cat V$ and abbreviate this notation to $\cat{Op}$ in the case $\cat V=\cat{sSet}$.

\begin{example}
The \emph{unit operad} $\op J$ is the unique operad in $\cat V$ with \[\op J(I)=\begin{cases}
1_{\cat V}&:\quad |I|=1\\
\varnothing &:\quad |I|\neq 1.
\end{cases}\] Any other operad $\op O$ in $\cat V$ receives a canonical map of operads from $\op J$.
\end{example}

We pause to establish some notation that will be useful in Section \ref{sec:bvt} below.

\begin{notation} For finite sets $I$ and $J$ and a symmetric sequence $\op X$, we write $\pi_i:\{i\}\times J \to J$ for the projection and $\pi_i^*: \op X(J)\to \op X( \{i\} \times J)$ for the induced map. Similarly, we write $\bar\pi_j:I\times \{j\} \to I$ for the projection and $\bar \pi_j^*: \op X(I)\to \op X( I \times \{j\})$ for the induced map.
\end{notation}

\begin{remark}\label{rmk:comp-prod} The summand of $(\op Y \circ \op X)(I\times J)$ corresponding to the projection $I\times J \to J$ is
\[\op Y(J) \otimes_{\Sigma_J} \bigotimes_{j\in J} \op X\big( I\times \{j\}\big)\cong \op Y(J) \otimes_{\Sigma_J} \bigotimes_{j\in J} \op X\big( I\big),\]
where the isomorphism is induced by the maps $\bar\pi_j^*$.  Similarly, the summand of $(\op X \circ \op Y)(I\times J)$ corresponding to the projection $I\times J \to I$ is
\[\op X(I) \otimes_{\Sigma_I} \bigotimes_{i\in I} \op Y\big( \{i\}\times J\big)\cong \op X(I) \otimes_{\Sigma_I} \bigotimes_{i\in I} \op Y\big( J\big).\]
\end{remark}

We turn now to the theory of right modules over an operad $\op O$, which can be described as enriched presheaves on a certain $\cat V$-category associated to $\op O$---see \cite[Section 3]{arone-turchin}, for example. For any subcategory $\F'$ of $\F$, there is a $\cat V$-category $\F'(\op O)$ with objects the objects of $\F'$, hom objects given by \[\Hom_{\F'(\op O)}(I,J)= \coprod_{f\in \F'(I,J)} \bigotimes_{j\in J}\op O \big( f^{-1}(j)\big),\] and composition defined using composition in $\F'$ and the operad structure of $\op O$. A map of operads $\vp\colon  \op O \to \op P$ induces a $\cat V$-functor $\F'(\op O) \to \F'(\op P)$ covering the identity on $\F'$, which we also denote by $\vp$. Moreover, the inclusion $\F' \hookrightarrow \F$ induces a $\cat{V}$-functor $\F' (\op O) \to \F(\op O)$ for every operad $\op O$.

\begin{definition}
Let $\op O$ be an operad in $\cat V$. A right \emph{$\op O$-module} is a $\cat V$-functor $\F(\op O)^{op}\to \cat V$.
\end{definition}

We write $\Mod{\op O}$ for the $\cat V$-category of right $\op O$-modules. We typically omit the adjective ``right,'' as no other type of module will enter the discussion.

\begin{remark}
An obvious modification of this definition leads to a notion of $\op O$-module in any $\cat V$-category $\op C$. We shall have no use for such generality.
\end{remark}

\begin{example}
In the case $\op O=\op J$ of the unit operad, a $\op J$-module is simply a symmetric sequence.
\end{example}

Suppose now that $\cat V$ is $\cat{V}$-bicomplete. For any morphism of operads $\vp: \op O\to \op P$, there is an induced $\cat V$-adjunction
\[\adjunct {\Mod{\op O}}{\Mod{\op P},}{\vp_!}{\vp^*}\]
where $\vp^*$ is precomposition with $\vp$ and $\vp_!$ is enriched left Kan extension along $\vp$. 

\begin{example} When $\eta: \op J \to \op O$ is the unit map of the operad $\op O$, then the adjunction 
\[\adjunct {\Mod{\op J}}{\Mod{\op O}}{\eta_!}{\eta^*}\]
is the usual free-forgetful adjunction for $\op O$-modules.
\end{example}

It will be important for us in what follows to work with a more restrictive notion of free module. Write $\iota : \cat{N}=\cat{N}(\op J) \hookrightarrow \F (\op J)$ for the inclusion functor, and set $\mathbb F_{\op O}=\eta_!\circ\iota_!$ and $\mathbb U_{\op O}=\iota^*\circ\eta^*$.

\begin{definition} An $\op O$-module $\op M$ is \emph{totally free} if it lies in the essential image of $\mathbb{F}_{\op O}$.
\end{definition}

\begin{remark}
Take $\cat V=\cat{sSet}$. The category $\cat{sSet}^{\cat{N}(\op J)}$ is simply the category of sequences in $\cat{sSet}$, and the functor $\iota_!$ sends a sequence $(C_n)_{n\in \mathbb N}$ to $(C_n\otimes \Sigma_n)_{n\in \mathbb N}$, where $\otimes$ denotes simplicial tensoring.
\end{remark}

The $(\mathbb{F}_{\op O}, \mathbb{U}_{\op O})$-adjunction witnesses $\op O$-modules as monadic over sequences (not over \emph{symmetric} sequences). The corresponding comonad $\mathbb{K}_{\op O}$ will be important in the following section in defining a version of the bar construction.

\subsection{Homotopy theory of operadic modules}

Under reasonable conditions, $\Mod{\op O}$ inherits a model structure from $\cat V$, as established in \cite[Section 7]{moser}.  We are particularly interested in the following cases.

\begin{proposition}\label{prop:module model cat} Let $\cat V$ denote either $\cat {sSet}$ or $\cat{Ch}_R$, equipped with the Kan-Quillen or projective model structure, respectively. For every $\op O\in \cat {Op}(\cat V)$, the category $\Mod{\op O}$ admits  the \emph{projective model structure} inherited from $\cat V$, in which weak equivalences and fibrations are defined objectwise.
This is a proper $\cat V$-model structure.
\end{proposition}

\begin{proof} See Examples 7.6 and 7.8 and Propositions 8.1. and 8.2 in \cite{moser}.
\end{proof}

It is easy to see that if $\V$ is either $\cat {sSet}$ or $\cat{Ch}_R$, then, for any morphism of operads $\vp: \op O\to \op P$, the adjunction
\[\adjunct{\Mod{\op O}}{\Mod{\op P}}{\vp_!}{\vp^*}\]
is a Quillen adjunction with respective the projective model structures. 

For the rest of this subsection, we restrict to $\V=\cat {sSet}$.

\begin{definition} Let $\op O$ be a simplicial operad. The \emph{sequential $\op O$-bar construction} is the composite \[B_{\cat N}^{\op O}:\Mod{\op O}\xrightarrow{B_\bullet^{\mathbb K_{\op O}}} \Mod{\op O}^{\cat \Delta ^{\mathrm{op}}}\xrightarrow{|-|}\Mod{\op O}\] of the simplicial comonadic bar construction with geometric realization.
\end{definition}

\begin{warning}
The reader is urged not to confuse our sequential bar construction, which is premised on viewing $\op O$-modules as monadic over sequences, with the more standard bar construction premised on viewing $\op O$-modules as monadic over \emph{symmetric sequences}.
\end{warning}

\begin{lemma}\label{lem:bar cofibrant} Let $\op O$ be a simplicial operad. For any $\op O$-module $\op M$, the natural map $B^{{\op O}}_{\cat{N}}(\op M)\to \op M$ is a cofibrant replacement in $\Mod{\op O}$.
\end{lemma}
\begin{proof}That the map in question is a weak equivalence is an immediate consequence of \cite[Prop.~3.13]{JohnsonNoel:LHTAMSM}, since $\mathbb{U}_{\op O}$ preserves colimits and since $\Mod{\op O}$ carries the projective model structure so that \cite[Def.~3.1]{JohnsonNoel:LHTAMSM} applies, i.e., the monad $\mathbb{U}_{\op O}\mathbb{F}_{\op O}$ is simplicial Quillen. 

To see that $B^{\op O}_{\cat{N}}(\op M)$ is a cofibrant $\op O$-module, it suffices to show that $B_\bullet^{\mathbb K_{\op O}}(\op M)$ is a Reedy cofibrant simplicial $\op O$-module. For this, we proceed, as in \cite[Sec. 3.3]{JohnsonNoel:LHTAMSM}, by noting that the restriction of $B^{\op O}_{\cat{N}}(\op M)$ to the wide subcategory $\cat\Delta_0^{op}\subseteq \cat\Delta^{op}$ of the order-preserving functions preserving the minimal element is isomorphic to a diagram of the form $\mathbb{F}_{\op O}(X_\bullet)$, where $X_\bullet:\cat\Delta^{op}_0\to \cat{sSet}$ is given on objects by $X_n=\mathbb{U}_{\op O}\mathbb{K}_{\op O}^{\circ n}(\op M)$. This diagram satisfies the hypotheses of \cite[Prop.~3.22]{JohnsonNoel:LHTAMSM}, so the desired conclusion follows by applying \cite[Prop.~3.17]{JohnsonNoel:LHTAMSM} to $B_\bullet^{\mathbb K_{\op O}}(\op M)$.
\end{proof}

\subsection{Tensor products of operads and modules}\label{sec:bvt} In this section, we take $\cat V=\cat{sSet}$. The Boardman--Vogt tensor product of simplicial operads, denoted here by $\star$, codifies interchanging algebraic structures \cite{boardman-vogt:lnm}. That is, for all $\op O, \op P\in \cat {Op}$, a $(\op O\star \op P)$-algebra can be viewed as a $\op O$-algebra in the category of $\op P$-algebras or as a $\op P$-algebra in the category of $\op O$-algebras. (We provide this equivalent description of $\op O\star \op P$ for readers familiar with algebras over operads, but do not define the term ``algebra'' here, since we shall have no use for it.)

\begin{definition}[\label{definition:bv-op}\cite{boardman-vogt:lnm}] The \emph{Boardman--Vogt tensor product} of simplicial operads $\op O$ and $\op P$ is the simplicial operad $\op O\bvt \op P$ given by the quotient of the coproduct of $\op O$ and $\op P$ in $\cat{Op}$ by the equivalence relation generated by
\begin{equation}\label{eqn:bv}
\left(o, \left(\pi_i^*(p)\right)_{i\in I}\right)\sim \left(p, \left(\bar\pi_j^*(o)\right)_{j\in J}\right)
\end{equation}
for all $o\in \op O(I)$, $p\in \op P(J)$, and $I,J\in \cat F$.
\end{definition}

\begin{notation} Let $o\star p$ denote the equivalence class of $\left(o, \left(\pi_i^*(p)\right)_{i\in I}\right)$ and $\left(p, \left(\bar\pi_j^*(o)\right)_{j\in J}\right)$ in $(\op O\bvt \op P)(I\times J)$.
\end{notation}

Note that, in light of Remark \ref{rmk:comp-prod}, the lefthand side of (\ref{eqn:bv}) is an element of $(\op O\circ \op P)(I\times J)$ and the righthand side an element of $(\op P\circ \op O)(I\times J)$. To make sense of this definition, we use that $\op O\circ \op P$ and $\op P\circ \op O$ are both summands (modulo identification of identity elements) of the coproduct in $\cat {Op}$ of $\op O$ and $\op P$.

We now explain how to lift the Boardman--Vogt tensor product from simplicial operads to operadic modules. For concreteness, we consider only modules in $\cat{sSet}$, but our framework may be adapted with ease to include other simplicial targets equipped with suitable monoidal structures.

Write $\nu:\cat{F} \times \cat F \to \cat F$ for the Cartesian product of finite sets, and fix a subcategory $\F'\subseteq \F$ closed under finite products. Recall that, for a simplicial operad $\op O$ and finite sets $I$ and $J$, a $p$-simplex of $\map_{\F'(\op O)}(I,J)$ is a pair $\big(f, (o_{j})_{j\in J}\big)$, where $f: I\to J$ is an arrow in $\cat F'$ and $o_{j}$ is a $p$-simplex of $\op O\big(f^{-1}(j)\big)$. If $\op P$ is another operad, there is a simplicial functor
\[\mu: \F'(\op O) \times \F' (\op P) \to \F' (\op O\bvt \op P)\] covering $\nu$ and natural in $\op O$ and $\op P$. Explicitly, $\mu$ is defined on objects by $\mu(I,I')= I\times I'$ and on simplicial hom sets as the map
\begin{align*}\map_{\F'(\op O)}(I,J)\times \map_{\F'(\op P)}(I',J') &\to \map_{\F'(\op O\bvt \op P)}(I\times I', J\times J')\\
\big( \big(f, (o_{j})_{j\in J}\big), \big(g, (p_{j'})_{j'\in J'}\big)\Big)&\mapsto \Big( f\times g, (o_j\star p_{j'})_{(j,j')\in J\times J'}\Big).
\end{align*}
To see that $\mu$ is indeed a functor, note that the construction $(o,p)\mapsto o\star p$ determines a map from the matrix product of the underlying symmetric sequences of $\op O$ and $\op P$ to $\op O\star\op P$.  The Boardman-Vogt tensor product is define precisely so that this is a map of operads in each variable separately, from which follows easily that $\mu$ is functor.

The following definition is a mild generalization of the one given in \cite{DwyerHessKnudsen:CSP} following \cite{dwyer-hess:bv}. Our notation here, which differs from those references, is chosen to avoid confusion with various other tensor products of interest.

\begin{definition}\label{def:bv tensor} Let $\op O$ and $\op P$ be simplicial operads and $\F'\subseteq \F$ a subcategory closed under finite products. For $\op M\in \cat{sSet}^{\F'(\op O)^{\mathrm{op}}}$  and $\op N\in \cat{sSet}^{\F'(\op P)^{\mathrm{op}}}$, the \emph{Boardman--Vogt tensor product} of $\op M$ and $\op N$ is the enriched left Kan extension in the diagram of simplicial categories
\[\xymatrix{
\F'(\op O)^{op}\times\F'(\op P)^{\mathrm{op}}\ar[d]_\mu\ar[r]^-{\op M\times\op N}& \cat{sSet}\times\cat{sSet}\ar[rr]^-{-\times -}&&\cat{sSet}\\
\F'(\op O\bvt\op P)^{op}\ar@{-->}[urrr]_-{\op M\bvt\op N}
}\]
\end{definition}

In other words, $\op M\bvt \op N=\mu_!(\op M\boxtimes \op N)$, where 
$$-\boxtimes -: \cat{sSet}^{\F'(\op O)^{\mathrm{op}}}\times \cat{sSet}^{\F'(\op P)^{\mathrm{op}}} \to \cat{sSet}^{\F'(\op O)^{\mathrm{op}}\times \F'(\op P)^{\mathrm{op}}}$$
is the external product, specified by $(\op M\boxtimes \op N)(I,J)=\op M(I) \times \op M(J)$.

\begin{example}\label{ex:bvt-seq} When $\F'=\F_{\text{disc}}$, so that $\cat{sSet}^{\F'(\op O)^{\mathrm{op}}}$ is equivalent to the category of sequences in $\cat{sSet}$, the Boardman--Vogt tensor product takes a particularly simple form.  If $\op X$ and $\op Y$ are sequences, then
$$(\op X \bvt \op Y)(n)\cong\coprod_{lm=n} \op X(l)\times \op Y(m).$$ 
\end{example}

The Boardman--Vogt tensor product of presheaves is natural in the operad coordinate and in the $\F'$ coordinate, in the following sense.

\begin{lemma}\label{lem:nat-bv} Let $\F'$ be a subcategory of $\F$ that is closed under finite products. Let $\vp: \op O \to \op O'$ and $\psi:\op P \to \op P'$ be morphisms of simplicial operads. For all $\op M\in \cat {sSet} ^{\F'(\op O)^{\mathrm{op}}}$  and $\op N\in \cat {sSet} ^{\F'(\op P)^{\mathrm{op}}}$, there is a natural isomorphism
$$(\vp\bvt \psi)_!(\op M\bvt \op N) \cong \vp_!(\op M) \bvt \psi_!(\op N)$$
in $\cat {sSet}^{\F' (\op O' \bvt \op P')^{\mathrm{op}}}$.  Moreover, if $\iota: \F' \hookrightarrow \F$ denotes the inclusion functor, then there is a natural isomorphism
$$\iota_!(\op M\bvt \op N) \cong \iota_!(\op M) \bvt \iota_!(\op N)$$
in $\Mod {\op O\bvt \op P}$.
\end{lemma}

\begin{proof} Naturality of $\mu$ implies that the diagram
\[\xymatrix{\cat{sSet}^{\F'(\op O)^{\mathrm{op}}\times \F'(\op P)^{\mathrm{op}}} \ar [r]^(0.55){\mu_!} \ar[d]_{(\vp\times \psi)_!} & \cat{sSet}^{\F' (\op O \bvt \op P)^{\mathrm{op}}}\ar[d]^{(\vp\bvt \psi)_!}\\
\cat{sSet}^{\F'(\op O')^{\mathrm{op}}\times \F'(\op P')^{\mathrm{op}}} \ar [r]^(0.55){\mu_!}& \cat{sSet}^{\F' (\op O' \bvt \op P')^{\mathrm{op}}}}\]
commutes. Since formation of the Cartesian product with a fixed simplicial set preserves colimits, it follows from the colimit description of left Kan extensions that the diagram
\[\xymatrix{\cat{sSet}^{\F'(\op O)^{\mathrm{op}}}\times \cat{sSet}^{\F'(\op P)^{\mathrm{op}}} \ar [rr]^(0.5){-\boxtimes -} \ar[d]_{\vp_!\times \psi_!}&&\cat{sSet}^{\F'(\op O)^{\mathrm{op}}\times \F'(\op P)^{\mathrm{op}}} \ar[d]_{(\vp\times \psi)_!}\\
\cat{sSet}^{\F'(\op O')^{\mathrm{op}}}\times \cat{sSet}^{\F'(\op P')^{\mathrm{op}}} \ar [rr]^(0.5){-\boxtimes -} &&\cat{sSet}^{\F'(\op O')^{\mathrm{op}}\times \F'(\op P')^{\mathrm{op}}}}\]
also commutes, concluding the proof of the first isomorphism. We omit the argument for the second isomorphism, which is very similar.
\end{proof}

\begin{example}Applied to the unit maps $\eta_{\op O}:\op J \to \op O$ and $\eta_{\op P}:\op J \to \op P$, Lemma \ref{lem:nat-bv} implies that 
\[(\eta_{\op O})_!(\op X) \bvt (\eta_{\op P})_!(\op Y) \cong (\eta_{\op O\bvt \op P})_!(\op X\bvt \op Y)\] for all symmetric sequences $\op X$ and $\op Y$,
i.e., the Boardman--Vogt tensor product of a free $\op O$-module and a free $\op P$-module is the free $\op O\bvt \op P$-module on the Boardman--Vogt tensor product of the generating symmetric sequences.  
This isomorphism was first established in \cite{dwyer-hess:bv}, where the Boardman--Vogt tensor product of symmetric sequences was called the \emph{matrix monoidal product} and denoted $\Box$. 
\end{example}

\begin{example}\label{ex:bvt}
In the case $\F'= \cat N$ and $\op O=\op J =\op P$, Lemma \ref{lem:nat-bv} implies that for all seqences $\op X$ and $\op Y$, we have a natural isomorphism
\[\iota_!(\op X) \bvt \iota _!(\op Y) \cong \iota_!(\op X \bvt \op Y)\]
of symmetric sequences (we use that $\op J\bvt \op J\cong \op J$). Thus, applying Example \ref{ex:bvt-seq}, we have the natural isomorphism of $\op O\bvt \op P$-modules
\[\mathbb F_{\op O}(\op X) \bvt \mathbb F_{\op P}(\op Y) \cong \mathbb F_{\op O\bvt \op P}(\op X\bvt \op Y)\cong\mathbb{F}_{\op O\bvt \op P}\left(n\mapsto \coprod_{lm=n} \op X(l)\times \op Y(m)\right).\]
\end{example}

The lifted Boardman--Vogt tensor product behaves well with respect to colimits and cofibrations.

\begin{lemma}[{\cite[Lem. 3.2, Prop. 3.12]{DwyerHessKnudsen:CSP}}]\label{lem:bvt-leftquillen} Let $\op O$ and $\op P$ be simplicial operads.   
The functors 
\[\op M\bvt -: \Mod{\op P} \to \Mod{\op O\bvt \op P}\] and \[-\bvt \op N: \Mod{\op O} \to \Mod{\op O\bvt \op P}\]
are left adjoints
for any $\op O$-module $\op M$ and $\op P$-module $\op N$. If $\op M$ (resp. $\op N$) is cofibrant, then the left adjoint is a left Quillen functor.
\end{lemma}

\begin{remark} As Emily Riehl pointed out to the authors, it is likely that \[-\bvt - : \Mod{\op O}\times \Mod{\op P} \to \Mod{\op O\bvt \op P}\] is actually a Quillen bifunctor, in particular because the Boardman--Vogt tensor product of free modules is free.  Since we do not need this stronger result here, we leave its proof to the curious reader. 
\end{remark}

By Lemma \ref{lem:bvt-leftquillen}, it makes sense to speak of the \emph{derived Boardman--Vogt tensor product} of an $\op O$-module $\op M$ and a $\op P$-module $\op N$. By Lemma \ref{lem:bar cofibrant}, this derived tensor product is computed as
\[\op M\bvt^{\mathbb L} \op N = B_{\cat N}^{\op O}(\op M) \bvt B_{\cat N}^{\op P}(\op N).\]

\section{Linear modules}\label{sec:linear}

Fix a commutative, unital ring $R$. Throughout this section and the following two sections, homology is implicitly with $R$-coefficients, and tensor products are implicitly over $R$. We use the terminology $R$-\emph{flat} to refer to objects built from simplicial sets---simplicial operads or modules over them, for example---with $R$-flat homology.  We let $\grMod R$ denote the category of $\mathbb{Z}$-graded $R$-modules, which is a closed symmetric monoidal category with the respect to the graded tensor product.

Recall that the functor $H_*:\cat{sSet}\to \grMod{R}$ is lax monoidal; in particular, for every pair of simplicial sets $K$ and $L$, there is a natural map $H_*(K)\otimes H_*(L) \to H_*(K\times L)$, which is an isomorphism if $K$ or $L$ is $R$-flat. It follows that, if $\op O$ is a simplicial operad, then $H_*(\op O)$ is a operad in $\grMod R$.

\subsection{Linearization}

\begin{definition}
Let $\cat C$ be a simplicial category. The $R$-\emph{linearization} of $\cat C$ is the $\grMod{R}$-enriched category $\cat C_R$ where 
\begin{enumerate}
\item $\ob \cat C_R=\cat C$,
\item $\Hom_{\cat C_R}(X,Y)= H_*\big(\Hom_{\cat C}(X,Y)\big)$ for all objects $X$ and $Y$, and
\item composition is given by the composites of the form 
\[H_*\big (\Hom_{\cat C}(X,Y)\big)\otimes H_*\big(\Hom_{\cat C}(Y,Z)\big)\to H_*\big(\Hom_{\cat C}(X,Y)\times \Hom_{\cat C}(Y,Z)\big)\to H_*\big(\Hom_{\cat C}(X,Z)\big),\] 
where the first map is obtained from the lax monoidal structure of $H_*$, and the second from composition in $\cat C$.
\end{enumerate}
\end{definition}

This construction is functorial in the sense that a simplicial functor $\vp:\cat C \to \cat D$ induces a $\grMod{R}$-functor $\vp_R:\cat C_R \to \cat D_R$ coinciding with $\vp$ on objects.  In particular, a map $\vp: \op O\to \op P$ of simplicial operads induces a $\grMod{R}$-functor $\vp_R:\F'(\op O)_R \to \F'(\op P)_R$ for every subcategory $\F'\subseteq \F$. The $R$-linearization functor is lax monoidal in the sense that there is an enriched comparison functor 
\[\nabla: \cat C_R\otimes \cat D_R\to (\cat C\times\cat D)_R,\] where the objects  of $C_R\otimes \cat D_R$ are those of their Cartesian product, and the hom objects are the tensor products of the hom objects of the factors. This comparison functor $\nabla$ is often an isomorphism, e.g., whenever $\cat C$ or $\cat D$ has $R$-flat simplicial mapping spaces.

\begin{definition}
Let $\op O$ be a simplicial operad. An \emph{$R$-linear (right) $\op O$-module} is a $\grMod{R}$-functor $\F(\op O)_R^{op}\to \grMod{R}$.
\end{definition}

We let  $\Mod{\op O_R}$ denote $\grMod{R}^{\F(\op O)_R^{op}}$, the $\grMod{R}$-category of $R$-linear $\op O$-modules.

\begin{remark}
An obvious modification of this definition leads to a notion of $R$-linear $\op O$-module in any $\grMod{R}$-category $\cat C$. We shall have no use for such generality.
\end{remark}

As in the non-linear case, we have a functor that produces ``totally free" modules.

\begin{notation} Let $\iota : \cat N(\op J) \hookrightarrow \F (\op J)$ denote the inclusion functor of the discrete subcategory.  For any simplicial operad $\op O$ with unit map $\eta$, we let 
\[\mathbb F_{\op O_R}=(\eta_R)_!\circ (\iota_R)_!: \grMod{R}^{\cat N(\op J)^{\mathrm{op}}}\to \Mod{\op O_R},\]
while $\mathbb U_{\op O_R}$ denotes the corresponding forgetful functor and $\mathbb{K}_{\op O_R}$ the resulting comonad.
\end{notation}

Note that, contrary to what the notation might suggest, there is in general no operad called $\op O_R$. On the other hand, there is an operad in graded $R$-modules associated to $\op O$, namely $H_*(\op O)$, and the lax monoidal structure of $H_*$ supplies a canonical $\grMod{R}$-functor \[\xi:\F(H_*(\op O))\to \F(\op O)_R,\] 
that is the identity on objects and natural in $\op O$.

The easy proof of the next lemma is left to the reader. 

\begin{lemma}\label{lem:flat F}
Let $\op O$ be a simplicial operad. If $\op O$ is $R$-flat, then the canonical functor $\xi:\F(H_*(\op O))\to \F(\op O)_R$ is an isomorphism of $\grMod{R}$-categories.
\end{lemma}

Thus, in the $R$-flat case, the category $\Mod{\op O_R}$ coincides with the category of modules for $H_*(\op O)$, interpreted as an operad in $\grMod{R}$.

\begin{definition}
Let $\op O$ be a simplicial operad and $\op M$ an $\op O$-module. The $R$-\emph{linearization} of $\op M$ is the object $H_*(\op M)$ in $\Mod{\op O_R}$ given by the composite $R$-linear functor \[\F(\op O)_R^{op}\xrightarrow{\op M_R} \S_R\xrightarrow{H_*} \grMod{R}.\]
\end{definition}

Note that the first functor in the composition above involves applying homology to the simplicial hom sets of the categories in question, while the second is given by applying homology to the objects of the second category, which are themselves simplicial sets.

In the obvious way, $R$-linearization of modules extends to a functor, which is natural with respect to base change; that is, for every morphism of simplicial operads $\vp: \op O\to \op P$, the diagram
\[\xymatrix{
\Mod{\op P}\ar[d]_-{\vp^*}\ar@{->}[r]^-{H_*}& \Mod{\op P_R}\ar[d]^-{\vp_R^*}\\
\Mod{\op O}\ar[r]^-{H_*}&\Mod{\op O_R}
}\]
commutes. With flatness assumptions, $R$-linearization is also compatible with the ``extension of scalars" functor $\vp_!$. We  need only the following rudimentary case if this compatibility.

\begin{lemma}\label{lem:free module homology}
Let $\op O$ be a simplicial operad and $\op X$ a simplicial sequence. There is a natural transformation $\mathbb{F}_{\op O_R}(H_*(\op X))\to H_*(\mathbb{F}_{\op O}(\op X))$ that is an isomorphism provided either $\op O$ or $\op X$ is $R$-flat.
\end{lemma}
\begin{proof}
The map arises from the universal property of $\mathbb{F}_{\op O_R}$ applied to the map induced on homology by the inclusion of $\op X$ into $\mathbb{U}_{\op O}\mathbb{F}_{\op O}(\op X)$. Either flatness assumption implies that the K\"{u}nneth isomorphism holds.
\end{proof}

\subsection{Homotopy theory of linear operadic modules} We now situate $R$-linear $\op O$-modules in a homotopical context. 

Write $\cat{grCh}_R$ for the category of chain complexes of graded $R$-modules. Explicitly, an object of this category is a bigraded $R$-module $V=\bigoplus_{p,q\in\mathbb{Z}} V_{p,q}$ equipped with a differential that decreases $p$ and preserves $q$. Write $\cat{grCh}_R^{\geq0}$ for the subcategory of chain complexes $V$ such that $V_{p,q}=0$ for $p<0$. Note that we require non-negativity only in the direction of the chain grading. 

A version of the classical Dold--Thom correspondence asserts that the functor of \emph{normalized chains} witnesses an equivalence of categories \[N:\grMod{R}^{\cat\Delta^{op}}\xrightarrow{\simeq} \cat{grCh}_R^{\geq0}.\] An account of this correspondence at a suitable level of generality may be found in \cite[Thm. 1.2.3.7]{lurie2}, for example.   For any simplicial operad $\op O$, this equivalence extends to an equivalence of module categories 
\[N:(\grMod{R}^{\cat\Delta^{op}})^{\F (\op O)_R^{op}}\xrightarrow{\simeq} (\cat{grCh}_R^{\geq0})^{\F (\op O)_R^{op}}.\]
A slight variant of \cite[Example 7.9]{moser} shows that if we equip $\cat{grCh}_R^{\geq0}$ with the model structure in which fibrations are surjections in positive degrees and weak equivalences quasi-isomorphisms, then $(\cat{grCh}_R^{\geq0})^{\F (\op O)_R^{op}}$ admits the projective model structure, which transfers via the equivalence $N$ to $(\grMod{R}^{\cat\Delta^{op}})^{\F (\op O)_R^{op}}\cong \Mod{\op O_R}^{\cat\Delta^{op}}$.

A bar construction serves as a cofibrant replacement in the $R$-linear context as well.  
 
\begin{definition}
Let $\op O$ be a simplicial operad. The \emph{$R$-linear sequential $\op O$-bar construction} is the composite \[B_{\cat N}^{\op O_R}:\Mod{\op O_R}^{\cat\Delta^{op}}\xrightarrow{B_\bullet^{\mathbb K_{\op O_R}}} \Mod{\op O_R}^{\cat\Delta^{op}\times\cat \Delta ^{\mathrm{op}}}\xrightarrow{|-|}\Mod{\op O_R}^{\cat\Delta^{op}}\] of the simplicial comonadic bar construction with geometric realization.
\end{definition}
 
\begin{lemma}\label{lem:bar cofibrant rlin} Let $\op O$ be an $R$-projective simplicial operad and $\op M_\bullet$ a simplicial $R$-linear $\op O$-module. If $\op M_\bullet$ is $R$-projective in each simplicial degree, then the natural augmentation $B_{\cat N}^{{\op O_R}}(\op M_\bullet)\to \op M_\bullet$ is a cofibrant replacement in $\Mod{\op O_R}^{\cat\Delta^{op}}$.
\end{lemma}
 \begin{proof}
The proof is identical to that of Lemma \ref{lem:bar cofibrant}. The assumptions on $\op O$ and $\mathbb{U}_{\op O_R}(\op M_\bullet)$ are necessary to verify the hypotheses of \cite[Prop. 3.22]{JohnsonNoel:LHTAMSM} (the corresponding assumption in the non-linear case is always satisfied).
\end{proof}
 
The two bar constructions admit the following comparison, which is easy to verify (see Lemma \ref{lem:free module homology}).
 
\begin{lemma}\label{lem:bar vs linear bar}
There is a natural transformation $B_{\cat N}^{{\op O_R}}\circ H_*\to H_*\circ B_{\cat N}^{{\op O}}$ of functors from $\Mod{\op O}$ to $\Mod{\op O_R}^{\cat\Delta^{op}}$ that is an isomorphism provided $\op O$ is $R$-flat.
\end{lemma}

\subsection{Simplicial linear vs. differential graded modules}\label{section:dgification} Throughout this section, we fix an $R$-flat simplicial operad $\op O$. By Lemma \ref{lem:flat F}, the category $\Mod{\op O_R}$ coincides with the category of modules obtained by viewing $H_*(\op O)$ as an operad in $\grMod{R}$. In this situation, we may compare the homotopy theory of $\Mod{\op O_R}^{\cat\Delta^{op}}$ with another natural homotopy theory associated to $H_*(\op O)$, namely that of the category of modules in $\Ch{R}$ obtained by viewing $H_*(\op O)$ instead as an operad in  $\Ch{R}$.  We denote this category  $\dgMod{H_*(\op O)}$ to avoid ambiguity.

To compare the two, we use the total complex functor $T:\cat{grCh}_R\to \Ch{R}$. There is a natural isomorphism
\[{T}\circ N\circ \mathbb{U}_{\op O_R}\circ \mathbb{F}_{\op O_R} \cong \mathbb{U}_{H_*(\op O)}\circ\mathbb{F}_{H_*(\op O)}\circ T\circ N,\] where $N$ denotes the functor of normalized chains, as above, enabling us to formulate the following definition.

\begin{definition}
The \emph{dg-ification functor} is the unique dashed filler making both of the following square diagrams commute.
\[\xymatrix{
\Mod{\op O_R}^{\cat\Delta^{op}}\ar@{-->}[rrr]\ar[d]^-{\mathbb{U}_{\op O_R}}&&&\dgMod{H_*(\op O)}\ar[d]_-{\mathbb{U}_{H_*(\op O)}}\\
(\grMod{R}^{\cat\Delta^{op}})^{\cat N(\op J)^{\mathrm{op}}}\ar@/^1.2pc/[u]^-{\mathbb{F}_{\op O_R}}\ar[r]^-N&(\cat{grCh}_R^{\geq0})^{\cat N(\op J)^{\mathrm{op}}}\,\ar@{^{(}->}[r]^-i&\cat{grCh}_R^{\cat N(\op J)^{\mathrm{op}}}\ar[r]^-{T}&\Ch{R}^{\cat N(\op J)^{\mathrm{op}}}\ar@/_1.2pc/[u]_-{\mathbb{F}_{H_*(\op O)}}
}\]
\end{definition}

It is not hard to construct the dg-ification functor, since we know how it should be defined on free modules, and every module is a coequalizer of free modules  of a type preserved by the forgetful functors \cite[Prop.~3.7]{Barr-Wells:TTT}. 

The only fact about dg-ification relevant for our purpose is the following.

\begin{proposition}\label{prop:dg left quillen}
For any $R$-flat simplicial operad $\op O$, dg-ification preserves and reflects homotopy colimits.
\end{proposition}
\begin{proof}
Homotopy colimits of operadic modules are preserved and reflected by the relevant forgetful functor, as the forgetful functor preserves and reflects both weak equivalences, and colimits of modules are created in the underlying category of sequences.  it suffices to show that each of the bottom functors preserves and reflects homotopy colimits. Each of these functors preserves and reflects weak equivalences, so it suffices to demonstrate mere preservation of homotopy colimits. 

It is well known that the normalized chains functor preserves homotopy colimits. The functor $T$ also preserves homotopy colimits because it is a left Quillen functor, since the right adjoint sends a chain complex $V$ to the graded chain complex given in auxiliary degree $q$ by $V[-q]$, a construction that preserves surjections and quasi-isomorphisms. Finally, the inclusion $i$ of non-negatively graded complexes preserves weak equivalences, hence is its own total left derived functor, and, since $i$ also preserves colimits, this left derived functor preserves homotopy colimits.
\end{proof}

\section{The operadic K\"{u}nneth spectral sequence}\label{sec:KSS}

\subsection{Linearized Boardman--Vogt tensor products}\label{section:linearized tensor product}

The constructions and results of Section \ref{sec:bvt} generalize in a straightforward way to the $R$-linear context. To begin, for any subcategory $\F'$ of $\F$ that is closed under products, let $\mu_R$ denote the $R$-linearization of $\mu: \F' (\op O) \times \F' (\op P) \to \F' (\op O\bvt \op P)$.

\begin{definition}
Let $\op O$ and $\op P$ be simplicial operads and $\F'\subseteq \F$ a subcategory closed under finite products. The \emph{linearized Boardman--Vogt tensor product} of $\op M\in \grMod{R}^{\F'(\op O)_R^{\mathrm{op}}}$ and $\op N\in \grMod{R}^{\F'(\op P)_R^{\mathrm{op}}}$ is the enriched left Kan extension in the following diagram of $\grMod{R}$-categories.
\[\xymatrix{
\F'(\op O)_R^{op}\otimes\F'(\op P)_R^{op}\ar[d]_\nabla\ar[r]^-{\op M\otimes \op N}&\grMod{R}\otimes\grMod{R}\ar[r]^-{\otimes}&\grMod{R}\\
\big(\F'(\op O)\times\F'(\op P)\big)_R^{op}\ar[d]_{\mu_R}\\
\F'(\op O\bvt\op P)_R^{op}\ar@{-->}[uurr]_-{\op M\bvt\op N}.
}\]
\end{definition}

\begin{remark}
The definition of the linearized Boardman--Vogt tensor product extends in an obvious way to $R$-linear modules valued in $R$-linear categories with a compatible symmetric monoidal structure. We will have no use for such generality.
\end{remark}

Naturality of the linearized Boardman--Vogt tensor product of presheaves in the operad and $\F'$ coordinates can be established by a proof essentially identical to that of Lemma \ref{lem:nat-bv}.

\begin{lemma}\label{lem:nat-bv-lin} Let $\F'$ be a subcategory of $\F$ that is closed under finite products. Let $\vp: \op O \to \op O'$ and $\psi:\op P \to \op P'$ be morphisms of simplicial operads. For all $\op M\in \grMod{R}^{\F'(\op O)_R^{\mathrm{op}}}$  and $\op N\in \grMod{R}^{\F'(\op P)_R^{\mathrm{op}}}$, there is a natural isomorphism
$$(\vp_R\bvt \psi_R)_!(\op M\bvt \op N) \cong (\vp_R)_!(\op M) \bvt (\psi_R)_!(\op N)$$
in $\grMod{R}^{\F' (\op O' \bvt \op P')_R^{\mathrm{op}}}$.  Moreover, if $\iota: \F' \hookrightarrow \F$ denotes the inclusion functor, then there is a natural isomorphism
$$\iota_!(\op M\bvt \op N) \cong \iota_!(\op M) \bvt \iota_!(\op N)$$
in $\Mod {(\op O\bvt \op P)_R}$.
\end{lemma}

\begin{example}\label{ex:bvt-lin} As in the non-linear case, Lemma \ref{lem:nat-bv-lin} implies that for all sequences $\op X$ and $\op Y$ in $\grMod{R}$
$$\mathbb F_{\op O_R}(\op X) \bvt \mathbb F_{\op P_R}(\op Y) \cong \mathbb F_{(\op O\bvt \op P)_R}(\op X\bvt \op Y)$$
in $\Mod {(\op O\bvt \op P)_R}$.
\end{example}

There is a comparison map between the homology of Boardman--Vogt tensor products of sequences and the Boardman--Vogt tensor product of their homologies.

\begin{lemma}\label{lem:box tensor flat} Let $\op X$ and $\op Y$ be sequences of simplicial sets.
There is a natural map $H_*(\op X)\bvt H_*(\op Y)\to H_*(\op X\bvt\op Y)$ of sequences of graded $R$-modules, which is an isomorphism if either $\op X$ or $\op Y$ is $R$-flat.
\end{lemma}
\begin{proof} The component of the desired map in arity $k$ is the map 
\[\left(H_*(\op X)\mat H_*(\op Y)\right)(k)=\bigoplus_{ij=k}H_*(\op X(i))\otimes H_*(\op Y(j))\to H_*\left(\coprod_{ij=k}\op X(i)\times \op Y(j)\right)=H_*\left((\op X\mat\op Y)(k)\right)\] induced by the lax monoidal structure and compatibility with coproducts of $H_*$. The hypothesis of $R$-flatness ensures that the K\"{u}nneth isomorphism holds.
\end{proof}

\begin{lemma}\label{lem:bv vs linear bv} Let $\op O$ and $\op P$ be simplicial operads.  For all $\op M\in \Mod{\op O}$ and $\op N\in \Mod{\op P}$,
there is a natural map $H_*(\op M)\bvt H_*(\op N)\to H_*(\op M\bvt \op N)$ of $R$-linear $\op O\bvt \op P$-modules, which is an isomorphism if $\op M$ and $\op N$ are totally free on $R$-flat generating sequences.
\end{lemma}

\begin{proof} For all pairs of modules $\op M$, $\op N$ as in the statement of the lemma, the respective lax monoidal structures of $H_*$ and of the tensor product of $\grMod{R}$-categories combine to produce a $\grMod R$-natural transformation 
\[\alpha: (-\otimes -)\circ (H_*(\op M)\otimes H_*(\op N)) \Rightarrow H_* \circ (-\times -)_R \circ \nabla \circ (\op M_R \otimes \op N_R).\]  
On the other hand, the canonical $\S$-natural transformation
$$(-\times -)\circ (\op M\times \op N) \Rightarrow (\op M \bvt \op N)\circ \mu$$ induces a $\grMod R$-natural transformation
$$\beta: H_* \circ (-\times -)_R \circ (\op M\times \op N)_R \Rightarrow H_* \circ (\op M \bvt \op N)_R \circ \mu_R = H_*(\op M \bvt \op N) \circ \mu _R.$$
By the universal property of enriched left Kan extensions, there is therefore a unique $\grMod R$-natural transformation 
$$H_*(\op M )\bvt H_*(\op N) \Rightarrow H_*(\op M\bvt \op N)$$
factoring the composite $\beta\alpha$.

When $\op M$ and $\op N$ are totally free, we may write $\op M=\mathbb{F}_{\op O}(\op X)$ and $\op N=\mathbb{F}_{\op P}(\op Y)$, where $\op X$ and $\op Y$  are seqences in $\cat {sSet}$. Taking these sequences to be $R$-flat, we have the isomorphism
\begin{align*}
H_*(\op M\bvt \op N)&=H_*(\mathbb{F}_{\op O}(\op X)\bvt \mathbb{F}_{\op P}(\op Y))\\
&\cong H_*(\mathbb{F}_{\op O\bvt \op P}(\op X\bvt \op Y)) &\quad (\ref{ex:bvt})\\
&\cong \mathbb{F}_{(\op O\bvt\op P)_R}(H_*(\op X\bvt\op Y)) &\quad (\ref{lem:free module homology})\\
&\cong \mathbb{F}_{(\op O\bvt\op P)_R}(H_*(\op X)\bvt H_*(\op Y))&\quad (\ref{lem:box tensor flat})\\
&\cong \mathbb{F}_{\op O_R}(H_*(\op X))\bvt \mathbb{F}_{\op P_R}(H_*(\op Y))&\quad (\ref{ex:bvt-lin})\\
&\cong H_*(\mathbb{F}_{\op O}(\op X))\bvt H_*(\mathbb{F}_{\op P}(\op Y))&\quad (\ref{lem:free module homology})\\
&\cong H_*(\op M)\bvt H_*(\op N).
\end{align*}
\end{proof}

As in the non-linear case, the linearized Boardman--Vogt tensor product with a fixed module behaves well.  The proofs of the properties below again follow immediately from results in \cite[Sec. 3]{DwyerHessKnudsen:CSP}.

\begin{lemma}\label{lem:bvt-leftquillen-lin}
Let $\op O$ and $\op P$ be simplicial operads.   
The functors \[\op M\bvt -: \Mod{\op P_R}\to \Mod{(\op O\bvt \op P)_R}\] and \[-\bvt \op N: \Mod{\op O_R}\to \Mod{(\op O\bvt \op P)_R}\]
are left adjoints
for any $R$-linear $\op O$-module $\op M$ and $R$-linear $\op P$-module $\op N$. If $\op M$ (respectively, $\op N$) is cofibrant, then the left adjoint is a left Quillen functor.
\end{lemma}

By Lemma \ref{lem:bvt-leftquillen-lin}, it makes sense to speak of the \emph{derived Boardman--Vogt tensor product} of an $R$-linear $\op O$-module $\op M$ and a $R$-linear $\op P$-module $\op N$. In good circumstances, Lemma \ref{lem:bar cofibrant} allows us to compute this derived tensor product as a tensor product of bar constructions.

\subsection{The spectral sequence}

The operadic K\"unneth spectral sequence is a special case of the following construction.

\begin{proposition}\label{prop:linearization spectral sequence} Let $\op O$ be a simplicial operad, and let $\op M_\bullet$ be a simplicial $\op O$-module. There is a natural, convergent spectral sequence of $R$-linear $\op O$-modules 
\[E^2_{p,q}\cong H_p\big(H_*(\op M_\bullet)\big)_q\implies H_{p+q}(|\op M_\bullet|).\]
\end{proposition}

\begin{proof}
The homology $H_*(|\op M_\bullet|)$ is computed by the total complex of the double complex obtained by first applying the singular chains functor $C_*(-;R)$ levelwise to $\op M_\bullet$ and then using the normalized complex functor $N_*$ of the Dold--Kan correspondence to pass from simplicial chain complexes to double complexes over $R$. The desired spectral sequence is one of the two spectral sequences associated to this bicomplex; specifically, it is the spectral sequence obtained by using the differential derived from the singular chains functor first and the simplicial differential second. This spectral sequence is concentrated in the first quadrant, hence convergent.

By naturality, this is a spectral sequence of $R$-linear $\op O$-modules.
\end{proof}

Specializing now to Boardman--Vogt tensor products of modules, we obtain a spectral sequence converging from the (derived) Boardman--Vogt tensor product of homology modules to the homology of the (derived) Boardman--Vogt tensor product.

\begin{theorem}\label{thm:general kuenneth}
Let $\op O$ and $\op P$ be simplicial operads, $\op M$ an $\op O$-module, and $\op N$ a $\op P$-module, and assume that all four are $R$-projective. There is a natural, convergent spectral sequence 
\[H_p\big(H_*(\op M)\bvt^\mathbb{L} H_*(\op N)\big)_q\implies H_{p+q}\left(\op M\bvt^\mathbb{L}\op N\right)\] of $R$-linear $\op O\bvt \op P$-modules.
\end{theorem}

In the statement above, we consider $H_*(\op M)$ and $H_*(\op N)$ as constant objects in $\Mod{\op O_R}^{\cat\Delta^{op}}$  and $\Mod{\op P_R}^{\cat\Delta^{op}}$, respectively.  Their derived Boardman--Vogt tensor product is an object in $\Mod{(\op O\bvt\op P)_R}^{\cat\Delta^{op}}$ that is, in general, not constant.   The external homology is computed by applying the normalized chains functor  to the simplicial, graded $R$-module, which gives rise to a chain complex in graded $R$-modules, then computing homology. This homology is graded by homological degree $p$ and internal degree $q$.

\begin{proof}[Proof of Theorem \ref{thm:general kuenneth}]
We observe that 
\begin{align*}
\op M \bvt^\mathbb{L}\op N&\simeq B^{\op O}_{\cat N}(\op M)\bvt B^{\op P}_{\cat N}(\op N)&\quad (\ref{lem:bar cofibrant})\\
&\cong \left||B_\bullet^{\mathbb K_{{\op O}}}(\op M)\bvt B_\bullet^{\mathbb K_{{\op P}}}(\op N)|\right|&\quad (\ref{lem:bvt-leftquillen})\\
&\cong |\d(B_\bullet^{\mathbb K_{{\op O}}}(\op M)\bvt B_\bullet^{\mathbb K_{{\op O}}}(\op N))|.
\end{align*}
Applying Proposition \ref{prop:linearization spectral sequence} to the simplicial $\op O\bvt\op P$-module $\d(B_\bullet^{\mathbb K_{{\op O}}}(\op M)\bvt B_\bullet^{\mathbb K_{{\op O}}}(\op N))$, we obtain a spectral sequence converging to the homology of $\op M\bvt^\mathbb{L}\op N$. It remains to identify the $E^2$-page of this spectral sequence. To do so, we note that 
\begin{align*}
H_*(\d(B_\bullet^{\mathbb K_{\op O}}(\op M)\bvt B_\bullet^{\mathbb K_{\op P}}(\op N)))&\cong \d(H_*(B_\bullet^{\mathbb{K}_{\op O}}(\op M))\bvt H_*(B_\bullet^{\mathbb{K}_{\op P}}(\op N)))&\quad (\ref{lem:bv vs linear bv})\\
&\cong \d\big(B_\bullet^{\mathbb{K}_{\op O_R}}(H_*(\op M))\bvt B_\bullet^{\mathbb{K}_{\op P_R}}(H_*(\op N)))&\quad (\ref{lem:bar vs linear bar})\\
&\simeq H_*(\op M)\bvt^\mathbb{L} H_*(\op N).&\quad (\ref{lem:bar cofibrant rlin})
\end{align*} Note that only the last step uses the assumption of projectivity.
\end{proof}

\section{Application to configuration spaces}

In this section, we combine the general considerations of Section \ref{sec:KSS} with the results of \cite{DwyerHessKnudsen:CSP} to prove Theorem \ref{thm:kuenneth ss}, which we deduce as a special case of a general result concerning products of manifolds equipped with tangential structures.

\subsection{Reminders on structured manifolds}
We review the conventions of \cite[4.1]{DwyerHessKnudsen:CSP}, following \cite[V.5-10]{andrade}, on manifolds equipped with tangential structures. 

\begin{definition}
Let $M$ be an $m$-manifold and $G\to GL(m)$ a continuous homomorphism, and write $\Fr_{M}$ for the frame bundle of the tangent bundle of $M$. A \emph{$G$-framing} on $M$ is a principal $G$-bundle $\Fr^G_M$ together with an isomorphism \[\varphi_M:\Fr^G_M\times_G GL(m)\xrightarrow{\cong} \Fr_M\] of principal $GL(m)$-bundles covering the identity. A \emph{framing} is a $G$-framing with $G$ trivial.
\end{definition}

We often abbreviate to $M$ the triple constituting a $G$-framed manifold, leaving all other data implicit. Examples of canonically $G$-framed manifolds include Euclidean $m$-space and disjoint unions and open subsets of $G$-framed manifolds. The Cartesian product of a $G$-framed manifold and an $H$-framed manifold is canonically $G\times H$-framed. Combining these examples, the configuration space $\Conf_k(M)\subseteq M^k$ is canonically $G^k$-framed whenever $M$ is $G$-framed.

\begin{definition}
The $G$-\emph{framed configuration space} of $k$ points in $M$ is the $G^k\rtimes\Sigma_k$-space \[\Conf^G_k(M):=\Fr_{\Conf_k(M)}^{G^k},\] where $\Sigma_k$ acts on $G^k$ by permuting the factors.
\end{definition}

The collection of $G$-framed manifolds forms a category under the following type of map.

\begin{definition}
Let $M_1$ and $M_2$ be $G$-framed manifolds. A $G$-\emph{framed embedding} consists of an embedding $f:M_1\to M_2$, a bundle map $\tilde f:\Fr_{M_1}^G\to \Fr_{M_2}^G$, and a $GL(m)$-equivariant homotopy $h:\Fr_{M_1}\times[0,1]\to \Fr_{M_2}$ from $Df$ to  the composite $\varphi_{M_2}\circ (\tilde f\times_GGL(m))\circ\varphi_{M_1}^{-1}$, where $Df:\Fr_{M_1}\to \Fr_{M_2}$ is the induced bundle map. We require that $\tilde f$ and $h$ each cover $f$.
\end{definition}

The set $\Emb^G(M_1, M_2)$ carries a natural topology in which composition is continuous, where the composite of $G$-framed embeddings is defined using composition of embeddings and bundle maps and \emph{pointwise} composition of homotopies. We denote the resulting topological category, which is symmetric monoidal under disjoint union, by $\Mfld^G_m$.

Formation of structured configuration spaces extends to a functor $\Conf_k^G:\Mfld_m^G\to \cat{Top},$ which is closely related to certain spaces of $G$-framed embeddings.

\begin{proposition}[{\cite[14.4]{andrade}}]\label{prop:embedding and conf}
For each $k\geq0$, the natural transformation \[\Emb^G(\amalg_k\mathbb{R}^m, -)\longrightarrow\Conf^G_k(-)\] induced by evaluation at the origin is a $G^k\rtimes \Sigma_k$-equivariant weak equivalence.
\end{proposition}

\subsection{Skew little cubes and configuration spaces} Denote by $\Lambda(m)\subseteq GL(m)$ the subgroup of diagonal matrices with positive entries. 

\begin{definition}[{\cite[Def. 4.9]{DwyerHessKnudsen:CSP}}]\label{def:dilation rep}
A \emph{dilation representation} is a continuous group homomorphism $\rho:G\to GL(m)$ such that $\mathrm{im}(\rho)=(\mathrm{im}(\rho)\cap O(m))\cdot \Lambda(m).$
\end{definition}

We fix a dilation representation $\rho:G\to GL(m)$, which is left implicit in the notation, and write $\Box^m:=(-1,1)^m\subseteq \mathbb{R}^m$ for the open $m$-cube of side-length 2 centered at the origin.

\begin{definition}[{\cite[Def. 4.11]{DwyerHessKnudsen:CSP}}]
A $G$-\emph{skew little cube} is a pair $(v,g)$ with $v\in\Box^m$ and $g\in G$ such that the formula $f_{v,g}(x)=\rho(g)x+v$ specifies an embedding $f_{v,g}:\Box^m\to \Box^m$. A \emph{little $m$-cube} is a $\Lambda(m)$-skew little cube.
\end{definition}

The space $\op {C}^G_m(k)$ of $k$-tuples of $G$-skew little cubes with pairwise disjoint images forms an operad, with $\op {C}^{\Lambda(m)}_m$ recovering the usual little $m$-cubes operad $\op C_m$.

\begin{theorem}[{\cite[Thm. 4.14]{DwyerHessKnudsen:CSP}}]\label{thm:cubes and euclidean}
Let $\rho:G\to GL(m)$ be a dilation representation. There is a canonical weak equivalence of operads \[\varphi:\op C_m^G\to \End_{\Mfld_m^G}(\mathbb{R}^m).\]
\end{theorem}

Using this map, we obtain a $\op C_m^G$-module $\op C_M^G:=\varphi^*\HHom_{\Mfld_m^G}(\mathbb{R}^m,M)$ organizing the homotopy types of the structured configuration spaces of the $G$-framed manifold $M$.

\subsection{Additivity and the main result}

Fix dilation representations $G\to GL(m)$ and $H\to GL(n)$. There are canonical maps of operads $\op C^G_m\to \op C_{m+n}^{G\times H}$ and $\op C^H_n\to \op C^{G\times H}_{m+n}$, and these two maps satisfy the interchange relations defining a map $\iota$ from the Boardman--Vogt tensor product.

\begin{conjecture}{\cite[Conf. 4.18]{DwyerHessKnudsen:CSP}}\label{conj:additivity}
Let $G\to GL(m)$ and $H\to GL(n)$ be dilation representations. The map \[\iota:\op C^G_m\otimes\op C_n^H\to \op C_{m+n}^{G\times H}\] is a weak equivalence.
\end{conjecture}

We view this conjecture as a statement of ``local additivity'' for configuration spaces. The main result of \cite{DwyerHessKnudsen:CSP} is the following global additivity statement.

\begin{theorem}{\cite[Thm. 5.7]{DwyerHessKnudsen:CSP}}
Let $G\to GL(m)$ and $H\to GL(n)$ be dilation representations, $M$ a $G$-framed $m$-manifold, and $N$ an $H$-framed $n$-manifold. If Conjecture \ref{conj:additivity} holds for $G$ and $H$, then there is a natural isomorphism 
\[\Ho( \iota^{*})(\op C_{M\times N}^{G\times H})\cong \op C_M^G\otimes^\mathbb{L} \op C_N^H\] in $ \Ho(\Mod{\op C_m^G\otimes\op C_n^H}).$
\end{theorem}

Combining this result with Theorem \ref{thm:general kuenneth}, we obtain the following consequence.

\begin{corollary}\label{cor:structured spectral sequence}
Let $G\to GL(m)$ and $H\to GL(n)$ be dilation representations, $M$ a $G$-framed $m$-manifold, and $N$ an $H$-framed $n$-manifold such that $G$, $H$, $\Conf^G_k(M)$ and $\Conf_\ell^H(N)$ are all $R$-projective, and assume that Conjecture \ref{conj:additivity} holds for $G$ and $H$. There is a natural, convergent spectral sequence \[E_{p,q}^2\cong H_p\left(H_*\left(\Conf^G(M);R\right)\star^\mathbb{L} H_*\left(\Conf^H(N);R\right)\right)_q\implies H_{p+q}\left(\Conf^{G\times H}(M\times N);R\right)\] of $R$-linear $\op C^{G\times H}_{m+n}$-modules.
\end{corollary}

Since Conjecture \ref{conj:additivity} is known to hold in the classical case of $G=\Lambda_m$ and $H=\Lambda_n$ \cite{dunn, brinkmeier}, the proof of Theorem \ref{thm:kuenneth ss} is complete.

\section{Formality and collapse}

In this section, we prove Theorem \ref{thm:collapse}, which asserts that the spectral sequence of Theorem \ref{thm:kuenneth ss} collapses in characteristic zero. In fact, we will show that collapse occurs in any situation in which the operads in question are formal.

We maintain our convention that homology and tensor products are taken with respect to a fixed commutative, unital ring $R$.

\subsection{Formality and Yoneda diagrams}

Throughout this section, $\op O$ denotes a fixed operad in $\Ch{R}$. The homology $H_*(\op O)$ is then an operad in both $\Ch{R}$ and $\grMod{R}$. In order to avoid confusion, we reflect this distinction in the notation for categories of modules.

\begin{definition}
We say that $\op O$ is \emph{formal} if there is a zig-zag \[\op O\xleftarrow{f} \widetilde{\op O}\xrightarrow{g} H_*(\op O)\] of weak equivalences of operads such that the induced automorphism of $H_*(\op O)$ is the identity.
\end{definition}

Although it appears stronger, this condition is equivalent to the existence of an isomorphism between $\op O$ and $H_*(\op O)$ in the homotopy category of operads in $\Ch{R}$. We record the following basic observation about the homology of modules over a formal operad.

\begin{lemma}\label{lem:formality and homology}
If $\op O$ is formal, and both $\op O$ and $H_*(\op O)$ are $R$-projective, then the diagram of functors \[\xymatrix{
\Ho(\Mod{\op O})\ar[dr]_-{H_*}\ar[r]^-{f^*}&\Ho(\Mod{\widetilde{\op O}})\ar[d]^-{H_*}\ar[r]^-{\mathbb{L}g_!}& \Ho(\Mod{H_*(\op O)}(\Ch{R}))\ar[dl]^-{H_*}\\
&\Mod{H_*(\op O)}(\grMod{R})
}\] commutes.
\end{lemma}
\begin{proof}
The commuting of the lefthand triangle is obvious. Under the assumption of projectivity, $\mathbb{L}g_!$ is an equivalence with inverse $g^*$ \cite[Thm. 16.B]{fresse}, so the righthand triangle commutes for the same reason.
\end{proof}

The goal of this section is to investigate a consequence of formality at the level of certain diagrams of $\op O$-modules.

\begin{definition}
Let $\cat I$ be a small category. A functor $F:\cat I\to \Mod{\op O}$ is called a \emph{Yoneda diagram} if it factors as \[F:\cat I\xrightarrow{} \F(\op O)\xrightarrow{Y_{\op O}} \Mod{\op O},\] where $Y_{\op O}$ denotes the enriched Yoneda embedding.
\end{definition}

\begin{proposition}\label{prop:yoneda hocolim}
Suppose that $\op O$ is formal and that both $\op O$ and $H_*(\op O)$ are $R$-projective. If $F:\cat I\to \Mod{\op O}$ is a Yoneda diagram, then the canonical map\[\hocolim_{\cat I} H_*(F)\to H_*(\hocolim_{\cat I} F)\] in $\Ho(\Mod{H_*(\op O)}(\Ch{R}))$ is an isomorphism.
\end{proposition}
\begin{proof} The main step in the proof is to construct a weak equivalence $H_*(F)\xrightarrow{\simeq} \mathbb{L}g_!f^*F $ of diagrams of $H_*(\op O)$-modules. Taking this task for granted momentarily, there results the sequence of isomorphisms \[H_*(\hocolim_{\cat I} H_*(F))\xrightarrow{\cong} H_*(\hocolim_{\cat I} \mathbb{L}g_!f^*F)\xrightarrow{\cong}H_*(\mathbb{L}g_!f^*\hocolim_{\cat I}F)\cong H_*(\hocolim_{\cat I} F)\] in $\Ho(\Mod{H_*(\op O)}(\Ch{R}))$, where the second uses that $f^*$ and $\mathbb{L}g_!$ both preserve homotopy colimits, and the third uses Lemma \ref{lem:formality and homology}. It follows that the standard spectral sequence, which converges from the leftmost term to the rightmost term, collapses, implying the claim.

We construct the desired map in the universal case of $\cat I=\F(\op O)$ and show that it is natural. Let $I$ and $J$ be finite sets and $\varphi\in \F(\op O)(I,J)$ an operation. Consider the following commuting diagram of $H_*(\op O)$-modules:
\[\xymatrix{
\mathbb{L}g_!f^*Y_{\op O}(I)\ar[rr]^-{\mathbb{L} g_!f^*Y_{\op O}(\varphi)}&& \mathbb{L}g_!f^*Y_{\op O}(J)\\
\mathbb{L} g_!Y_{\widetilde {\op O}}(I)\ar[d]\ar[u]\ar[rr]^-{\mathbb{L} g_!Y_{\widetilde{\op O}}(\varphi)}&& \mathbb{L} g_!Y_{\widetilde{\op O}}(J)\ar[d]\ar[u]\\
g_!Y_{\widetilde {\op O}}(I)\ar@{=}[d]_-\wr\ar[rr]^-{g_!f^*Y_{\op O}(\varphi)}&& g_!Y_{\widetilde{\op O}}(J)\ar@{=}[d]^-\wr\\
Y_{H_*(\op O)}(I)\ar@{-->}[rr]^-{\psi}&& Y_{H_*(\op O)}(J)
}\] (the map $\psi$ is defined by requiring the bottom square to commute). We abuse notation slightly in viewing the value of the derived functor $\mathbb{L}g_!$ as an object of the category of modules itself rather than of the homotopy category, for example by applying a cofibrant replacement functor pointwise.

Now, since $Y_{\widetilde{\op O}}$ takes values in cofibrant modules, and since $g_!$ is left Quillen, the middle vertical arrows are both weak equivalences. Since the top vertical arrows are weak equivalences by assumption, it remains to verify that $\psi=Y_{H_*(\op O)}\left([\varphi]\right)$, but our assumption on the induced automorphism of $H_*(\op O)$ implies that this equality holds up to homology in $\F(H_*(\op O))(I,J)$. Since this complex has trivial differential, the claim follows.
\end{proof}

\subsection{Configuration spaces from Yoneda diagrams}

In this section we complete the proof of Theorem \ref{thm:collapse}. The main step is to realize the module organizing the $G$-structured configuration spaces of a $G$-framed manifold as the homotopy colimit of a Yoneda diagram. For the sake of brevity, we write $\op E_m^G:=C_*(\End_{\Mfld_{m}^{G}}(\mathbb{R}^{m});R)$ and $\op E_M^G:=C_*(\HHom_{\Mfld_m^G}(\mathbb{R}^m,M);R)$, which we consider as an operad in $\Ch{R}$ and a module over that operad, respectively. Notice that $\op E_m^G$ is always $R$-projective, and $H_*(\op E_m^G)$ is $R$-projective provided $H_*(G)$ is so---this last claim follows from the K\"{u}nneth and universal coefficient theorems and the fact that $H_*(\Conf_k(\mathbb{R}^n);\mathbb{Z})$ is free Abelian for every $n$ and $k$ \cite[Lem. III.6.2]{cohen-lada-may}. 

\begin{notation} Given a topological category $\cat{C}$, we write $\cat{C}^\delta$ for the underlying discrete category---that is, the hom sets in $\cat{C}^\delta$ are the underlying sets of the hom spaces of $\cat{C}$. 
\end{notation}

Write $\Disk_m^G\subseteq \Mfld_m^G$ for the full topological subcategory spanned by those $G$-framed manifolds that are $G$-framed diffeomorphic to a (possibly empty) finite disjoint union of copies of $\mathbb{R}^m$ with its canonical $G$-structure. Since the discrete topology is initial, there is a canonical enriched functor $\cat{C}^\delta\to \cat{C}$. Thus, for a $G$-framed manifold $M$, there is a  composite functor \[\Disk_{m/M}^{G,\delta}\to \Disk_{m}^{G,\delta}\to \Disk_m^G\to \Mfld_{m}^G\xrightarrow{\op E_{(-)}^G} \Mod{\op E_{m}^G},\] which we abusively write as $\op E_{(-)}^G$, where the first functor in the sequence forgets the map to $M$.

\begin{lemma}\label{lem:hocolim over disks}
Let $G\to GL(m)$ and $H\to GL(n)$ be dilation representations, $M$ a $G$-framed and $N$ an $H$-framed manifold, and $R$ a commutative ring. \begin{enumerate}
\item The functor $\op E_{(-\times-)}^{G\times H}: \Disk_{m/M}^{G,\delta}\times \Disk_{n/N}^{H,\delta}\to \Mod{\op E_{m+n}^{G\times H}}$ is a Yoneda diagram.
\item The natural map \[\hocolim_{\Disk_{m/M}^{G,\delta}\times \Disk_{n/N}^{H,\delta}}\op E_{(-\times-)}^{G\times H}\to\op E_{M\times N}^{G\times H}\] is an isomorphism in $\Ho(\Mod{\op E_{m+n}^{G\times H}})$.
\end{enumerate}
\end{lemma}
\begin{proof}
The first claim follows easily from the definitions and the equivalence of simplicial categories $\Disk_{m+n}^{G\times H}\simeq \F(\op \End_{\Mfld_{m+n}^{G\times H}}(\mathbb{R}^{m+n}))$.

The second claim is essentially standard. In order not to duplicate efforts made elsewhere in the literature, we will pass briefly into an $\infty$-categorical (i.e., quasicategorical) context---see \cite{lurie} for a general reference on quasicategories. The reader is warned that our notation differs from that of our references.

Our first task is to explain the following commutative diagram:
\[\xymatrix{
\Disk(M)\times\Disk(N)\ar@/_3.0pc/[ddr]\ar[r]&\Disk_{m/M}^{G,\delta}\times \Disk_{n/N}^{H,\delta}\ar[d]\ar[r]&\Disk_{m/M}^{G,\infty}\times \Disk_{n/N}^{H,\infty}\ar@{-->}[ddl]\\
&\Disk_m^{G,\delta}\times \Disk_{n}^{H,\delta}\ar[d]_-{\Conf^{G\times H}_k(-\times-)}\\
&\cat{Top}^\infty
}\] We consider the ordinary categories appearing in this diagram as $\infty$-categories via the (suppressed) nerve functor, while the superscript $\infty$ indicates an $\infty$-category obtained from a topological category via the topological nerve functor \cite[Def.~1.1.5.5]{lurie}. The category $\Disk(M)$ is the partially ordered set of (possibly empty) unions of disjoint Euclidean neighborhoods in $M$ (resp. $N$). 

As for the functors, the unmarked vertical functor is the product of the projections from the overcategories, the righthand horizontal functor is the canonical one, the lefthand horizontal functor is given by a choice of a set of parametrizations of the Euclidean neighborhoods in $M$ and $N$, and the curved functor is defined by commutativity.

We claim that the natural map from the $\infty$-categorical colimit of the vertical composite to $\Conf^{G\times H}_k(M\times N)$ is an equivalence. To prove this claim, we will argue that this colimit agrees with the colimit of the curved functor, then prove the corresponding claim for the curved functor.

According to \cite[Prop.~2.19]{ayala-francis}, the righthand horizontal functor is a localization inverting the isotopy equivalences. Since configuration spaces are isotopy invariant, we conclude the existence of the dashed filler. Moreover, the same horizontal functor, as a localization, is final, so the colimit in question coincides with the colimit of the dashed functor \cite[Prop.~4.1.1.8]{lurie}. To conclude that this colimit coincides with that of the curved functor, it suffices to show that the horizontal composite is final, which follows by combining \cite[Prop.~5.5.2.13]{lurie2} with the proof of \cite[Prop.~3.9]{ayala-francis}.

We conclude the claim by noting that the $\infty$-categorical colimit coincides in the homotopy category with the homotopy colimit \cite[Thm.~4.2.4.1]{lurie}, and the natural map \[\hocolim_{\Disk(M)\times \Disk(N)}\Conf^{G\times H}_k(-\times-)\to \Conf^{G\times H}_k(M\times N)\] is a weak equivalence by a well-known hypercover argument (see \cite[Lem.~5.5]{DwyerHessKnudsen:CSP}, for example).

The lemma now follows, after a second invocation of \cite[Thm.~4.2.4.1]{lurie}, from the natural equivalences $\Emb^{G\times H}(\amalg_k\mathbb{R}^{m+n}, -\times -)\xrightarrow{\sim} \Conf_k^{G\times H}(-\times-)$, since the singular chains functor preserves weak equivalences and homotopy colimits.
\end{proof}

Note the special cases of $M=\pt$ and $N=\pt$, respectively.

\begin{hypothesis}\label{hyp:formality}
For the dilation representations $G\to GL(m)$ and $H\to GL(n)$ and the commutative ring $R$, the operads $\op E_m^G$, $\op E_n^H$, and $\op E_{m+n}^{G\times H}$ are formal.
\end{hypothesis}

We strongly emphasize that this hypothesis is \emph{not} a conjecture; indeed, it is known to fail in some cases \cite{moriya, Salvatore:PNFLDOCT}. On the other hand, by \cite{fresse-willwacher}, the hypothesis holds for any $m$ and $n$ as long as $G$ and $H$ are both contractible and $R$ is a field of characteristic zero---see also \cite{kontsevich,lambrechts-volic}. Thus, Theorem \ref{thm:collapse} is a consequence of the following more general conditional statement.

\begin{theorem}\label{thm:structured collapse}
Let $G\to GL(m)$ and $H\to GL(n)$ be dilation representations, $M$ a $G$-framed and $N$ an $H$-framed manifold, and $R$ a commutative ring. If Conjecture \ref{conj:additivity} and Hypothesis \ref{hyp:formality} hold for $G$, $H$, and $R$, and if $G$, $H$, $\Conf_k^G(M)$, and $\Conf_\ell^H(N)$ are all $R$-projective, then the spectral sequence of Corollary \ref{cor:structured spectral sequence} degenerates at $E^2$.
\end{theorem}
\begin{proof} To reduce notational clutter, we leave implicit all restriction functors along maps of operads. We explain the following sequence of isomorphisms in $\Ho\big(\Mod{(\op C^G_m\star \op C^H_n)_R}^{\cat\Delta^{op}}\big)$:
\begin{align*}
H_*\left(\op C^G_M\right)\star^\mathbb{L} H_*\left(\op C^H_N\right)&\xrightarrow{\simeq}H_*\left(\op E^G_M\right)\star^\mathbb{L} H_*\left(\op E^H_N\right)\\
&\xleftarrow{\simeq} H_*\left(\hocolim_{\Disk_{m/M}^{G,\delta}}\op E^G_{(-)}\right)\star^\mathbb{L} H_*\left(\hocolim_{\Disk_{n/N}^{H,\delta}}\op E^H_{(-)}\right)\\
&\xleftarrow{\simeq}\left(\hocolim_{\Disk_{m/M}^{G,\delta}}H_*\left(\op E^G_{(-)}\right)\right)\star^\mathbb{L}\left(\hocolim_{\Disk_{n/N}^{H,\delta}}H_*\left(\op E^H_{(-)}\right)\right)\\
&\xleftarrow{\simeq} \hocolim_{\Disk_{m/M}^{G,\delta}\times \Disk_{n/N}^{H,\delta}}H_*\left(\op E^G_{(-)}\right)\star^\mathbb{L}H_*\left(\op E^H_{(-)}\right)\\
&\xrightarrow{\simeq} \hocolim_{\Disk_{m/M}^{G,\delta}\times \Disk_{n/N}^{H,\delta}}H_*\left(\op E^{G\times H}_{(-\times-)}\right)\\
&\xrightarrow{\simeq} H_*\left(\hocolim_{\Disk_{m/M}^{G,\delta}\times \Disk_{n/N}^{H,\delta}}\op E^{G\times H}_{(-\times-)}\right)\\
&\xrightarrow{\cong} H_*\left(\op E^{G\times H}_{M\times N}\right)\\
&\xleftarrow{\cong} H_*\left(\op C^{G\times H}_{M\times N}\right).
\end{align*} The first and last follow from Theorem \ref{thm:cubes and euclidean}, the second and seventh from Lemma \ref{lem:hocolim over disks}(2), the fourth follows from Lemma \ref{lem:bvt-leftquillen-lin}, and the fifth from Lemma \ref{lem:bv vs linear bv} (note that the modules in question are totally free), Conjecture \ref{conj:additivity}, and \cite[Thm. 5.6(2)]{DwyerHessKnudsen:CSP}. For the remaining two, we argue as follows. Using the dg-ification functor introduced in Section \ref{section:dgification}, we may interpret the diagrams in question as diagrams of modules over the differential graded operads $H_*(\op E_m^G)$, $H_*(\op E_n^H)$, and $H_*(\op E_{m+n}^{G\times H})$, respectively. For these dg-ified diagrams, the desired equivalences follow, in light of our assumptions, from Proposition \ref{prop:yoneda hocolim} and Lemma \ref{lem:hocolim over disks}(1). Since dg-ification preserves and reflects homotopy colimits in our situation by Proposition \ref{prop:dg left quillen}, the claimed equivalences follow.
\end{proof}

\bibliographystyle{amsalpha} 
\bibliography{raao.bib}
\end{document}